\documentclass[11pt]{amsart}
\usepackage{amsmath}
\usepackage[all]{xy}

\usepackage{amssymb, amsmath, amsthm, mathrsfs}

 \newtheorem{thm}{Theorem}[section]
 \newtheorem{cor}[thm]{Corollary}
 \newtheorem{lem}[thm]{Lemma}
 \newtheorem{prop}[thm]{Proposition}

    \theoremstyle{definition}
 \newtheorem{defn}[thm]{Definition}
   \newtheorem{notn}[thm]{Assumption}

  \newtheorem{exam}[thm]{Example}
 \theoremstyle{remark}
 \newtheorem{rem}[thm]{Remark}
\numberwithin{equation}{section}

\newcommand {\de}{\operatorname {d}}
\newcommand {\vol}{\operatorname {vol}}

\begin{document}

\title[ Poisson double extensions]{ Poisson double extensions}

 \author{Qi Lou, Sei-Qwon Oh, S.-Q. Wang}
\address{School of Mathematical Sciences, Fudan University, Shanghai 200433, China}
\email{qlou@fudan.edu.cn}
\address{Department of Mathematics, Chungnam National  University, 99 Daehak-ro,   Yuseong-gu, Daejeon 34134, Korea}
\email{sqoh@cnu.ac.kr}
\address{Department of Mathematics, East China  University of Science and Technology, Shanghai 200237, China}
\email{sqwang@ecust.edu.cn}


\subjclass[2010]{17B63, 16S36}

\keywords{Poisson double extension, semiclassical limit, deformation quantization}



\begin{abstract}
A double Ore extension was introduced by James Zhang and Jun Zhang \cite{ZhZh} to study a class of Artin-Schelter regular algebras. Here we give a  definition of Poisson double extension which may be considered as an analogue of double Ore extension and show that algebras in a class of double Ore extensions are deformation quantizations of Poisson double extensions. We also investigate
 the modular derivations of Poisson double extensions and
 the relationship between Poisson double extensions and
iterated Poisson polynomial extensions. Results are illustrated by examples.
\end{abstract}

\maketitle


\section*{Introduction}
There are many kinds of noncommutative algebras appearing in noncommutative geometry, which are obtained from deforming certain  commutative algebras.  Artin-Schelter regular algebras \cite{ArSc}, which may be  considered as noncommutative analogues of  commutative polynomial rings, play an important role in noncommutative geometry and have been studied by many mathematicians.
In order to study  Artin-Schelter regular algebras of global dimension 4, James Zhang and Jun Zhang \cite{ZhZh} defined  a very useful concept ``double Ore extension" for constructing  Artin-Schelter regular algebras. The double Ore extension  may be considered as a generalization of an Ore extension \cite[\S1]{ZhZh}.
In \cite[Proposition 4.1]{LaLe}, Launois and  Lecoutre proved that a class of Ore extensions can be viewed as deformation quantizations of Poisson polynomial extensions, in other words, Poisson polynomial extensions correspond to Ore extensions by taking semiclassical limits in these cases. It arises a question whether  a double Ore extension is a deformation quantization of a Poisson algebra or not.
Here, we find  a class of Poisson algebras  such that their deformation quantizations are  double Ore extensions. That is, a main aim of this article is to prove that algebras in a class of double Ore extensions defined by James Zhang and Jun Zhang \cite{ZhZh} are deformation quantizations of certain Poisson algebras called Poisson double extensions and to characterize their Poisson structures. On the other side, the second author introduced the Poisson enveloping algebras in \cite{Oh5} which are closely related to Poisson algebras. In fact, for a polynomial Poisson algebra of $n$-variables, Zhu and Wang recently \cite{Zhu18} proved that its Poisson enveloping algebra can be characterized as a deformation quantization of a new polynomial Poisson algebra of $2n$-variables.  Moreover, for Poisson enveloping algebras, L$\ddot{\text{u}}$, Wang and  Zhuang  recently proved in \cite[Theorem 0.1]{LuWaZh}  that the Poisson enveloping algebra of a Poisson polynomial extension over a Poisson algebra $R$ is a double Ore extension of the Poisson enveloping algebra of $R$. Combining these results, one can convince oneself that the relationship (i.e., double Ore extension) between the two enveloping algebras should be  deformation quantization of a special Poisson extension (i.e., Poisson double extension).
By the way, L$\ddot{\text{u}}$, Wang, Yu and the second author considered the Poisson enveloping algebras of Poisson double extensions \cite{LuOhWaYu}. They proved that the Poisson enveloping algebra of a Poisson double extension  over a Poisson algebra $R$ is an iterated double Ore extension of the Poisson enveloping algebra of $R$.

Deformation quantization has been developed by many mathematicians, for instance,  Kontsevich \cite{Kon}. In this field, Poisson structures play an important role.  For example,
Bitoun \cite{Bit} and Van den Bergh \cite{Ber} studied algebraic properties of a filtered algebra  by constructing a  Poisson algebra induced from its filtration. In \cite{Ber1},
Van den Bergh computed the Hochschild homology of a class of Artin-Schelter regular algebras of dimension 3 by using the Poisson homology of the associated Poisson algebras.
In these cases, Poisson algebras induced from the given algebras can be used to detect many algebraic and homological properties of the given algebras. A secondary aim of this article is to give an algorithm to construct Poisson algebras induced from noncommutative algebras
 by reformulating their noncommutative relations. In fact, we  construct a class of Poisson algebras called Poisson double extensions induced from a class of double Ore extensions by interpreting the noncommutative relations in terms of Poisson brackets.
 As mentioned in \cite[\S1]{ChOh3}, let $h$ be a nonzero, nonunit, non-zero-divisor and central element of an algebra $R$ such that $R/hR$ is commutative.
Then $R/hR$ becomes a Poisson algebra with Poisson bracket
\begin{equation}\label{SCL}
\{\overline{a},\overline{b}\}=\overline{h^{-1}(ab-ba)}
\end{equation} for all
$\overline{a}=a+hR,\overline{b}=b+hR\in R/hR$. Following \cite[Chapter III.5]{BrGo},  the Poisson algebra $R/hR$ is called a {\it semiclassical limit} of $R$ and  $R$ is called a {\it quantization} of the Poisson algebra $R/hR$.
By a {\it deformation} of $R/hR$, we mean any ${\bf k}$-algebra of the form $R/(h-\lambda)R$, where $0\neq\lambda\in{\bf k}$ is such that the central element $h-\lambda$ is a nonunit in $R$. Here, we construct a class of Poisson algebras such that double Ore extensions are their deformation quantizations, and investigate their modular derivations and the relationship between   iterated Poisson polynomial extensions and them.

Let $t$ be an indeterminate and let ${\bf k}$ be a field of characteristic zero. In section 2, we consider a 5-tuple $(\Lambda,\Bbb F, R, A, t-1)$, where $\Lambda$ is a nonempty subset of ${\bf k}\setminus\{0,1\}$,  $\Bbb F$ is a subring  of the field ${\bf k}(t)$ containing ${\bf k}[t,t^{-1}]$, $t-1\in\Bbb F$, $R$ is an $\Bbb F$-algebra, $A$ is a right double Ore extension of $R$, such that $R$ and $A$ have  semiclassical limits $R_1:=R/(t-1)R$ and  $A_1:=A/(t-1)A$, respectively,  that $A_\lambda$, $\lambda\in\Lambda$, is a right double Ore extension of $R_\lambda$, and that $R_\lambda$ and $A_\lambda$ are  deformations of $R_1$ and $A_1$, respectively. (See Theorem~\ref{DPOE}.)
$$
\xymatrix{
  & &R\subset A \ar@{~>}[dll] \ar[drr]^-{\txt{\qquad semiclassical limit}}\\
 R_\lambda \subset A_\lambda \ar@{<--}[rrrr]^(.5){\txt{deformation}}& & & & R_1\subset A_1}
$$

We also analyze the Poisson bracket of $A_1$ and give a definition of Poisson double extension.
It is observed that the Poisson double extension
 may be considered as an analogue of double Ore extension and as a generalization of the Poisson polynomial extension in \cite{Oh8}. (See Theorem~\ref{LOW1}.)
 In section 3, we show that Poisson double  extensions preserve the smoothness of Poisson algebras, and give an explicit formula of modular derivations for Poisson double  extensions over affine smooth Poisson algebras with trivial canonical bundles. Section 4 is devoted to find a relationship between Poisson double extensions and iterated Poisson polynomial extensions as Cavalho, Lopes and Matczuk studied  a relationship  between double  extensions and iterated Ore extensions in \cite{CaLoMa}.
 In section 5, we give examples of Poisson double extensions. We find an example of a Poisson double extension that is not an iterated Poisson polynomial extension.
 In particular, we construct Poisson double extensions such that the double Ore extensions in \cite{ZhZh}, which are Artin-Schelter regular algebras,  are their deformation quantizations.

\medskip
Assume throughout the article that ${\bf k}$ denotes a field of characteristic zero, and that all vector spaces are over ${\bf k}$.
A {\it Poisson algebra} is a commutative ${\bf k}$-algebra $A$ with a Poisson bracket, which is a bilinear product
$\{-,-\}:A\times A\rightarrow A$ such that $A$ is a Lie algebra
under $\{-,-\}$, and that $H_a:=\{a,-\}$ is a derivation on $A$ for all $a\in A$.

\section{Double Ore Extensions}
Let us recall a right double Ore extension, shortly a right double extension, of an algebra $R$ defined in \cite[\S1]{ZhZh}. Let $\Bbb F$ be a commutative ${\bf k}$-algebra and $R$ be an $\Bbb F$-algebra. An $\Bbb F$-algebra $A$ containing $R$ as a subalgebra is said to be a {\it right double extension} of $R$ if
$A$ is generated by $R$ and two new variables $y_1,y_2$ such that
\begin{itemize}
\item $y_1$ and $y_2$ satisfy a relation
\begin{equation}\label{REL1}
y_2y_1=p_{11}y_1^2+p_{12}y_1y_2+\tau_1y_1+\tau_2y_2+\tau_0,
\end{equation}
where $P:=\{p_{11},p_{12}\}\subset\Bbb F$ and $\tau:=\{\tau_1,\tau_2, \tau_0\}\subset R$,
\item As a left $R$-module, $A$ is a free left $R$-module with a basis $\{y_1^iy_2^j|i,j\geq0\}$,
\item $y_1R+y_2R+R\subset Ry_1+Ry_2+R$.
\end{itemize}
Hence there exist $\Bbb F$-linear maps $\sigma_{11},\sigma_{12},\sigma_{21},\sigma_{22},\delta_1,\delta_2$ from $R$ to itself such that
\begin{equation}\label{REL2}
\begin{aligned}
y_1a&=\sigma_{11}(a)y_1+\sigma_{12}(a)y_2+\delta_1(a),\\
y_2a&=\sigma_{21}(a)y_1+\sigma_{22}(a)y_2+\delta_2(a)
\end{aligned}
\end{equation}
for all $a\in R$.  Set
$$\begin{aligned}
&y:=\left(\begin{matrix} y_1\\ y_2\end{matrix}\right)\in M_{2\times1}(A),&&\\
&\sigma:R\longrightarrow M_{2\times2}(R),&\ \ \ \sigma(a)&=\left(\begin{matrix}\sigma_{11}(a)&\sigma_{12}(a)\\ \sigma_{21}(a)&\sigma_{22}(a)\end{matrix}\right),\\
&\delta:R\longrightarrow M_{2\times1}(R),&\ \ \ \delta(a)&=\left(\begin{matrix}\delta_{1}(a)\\ \delta_{2}(a)\end{matrix}\right).
\end{aligned}$$
Note that $ M_{2\times1}(A)$, $M_{2\times2}(R)$ and $M_{2\times1}(R)$ are both left and right $R$-modules and that (\ref{REL2}) is expressed explicitly by
$$ya=\sigma(a)y+\delta(a)$$
for all $a\in R$. In this case, we say that the right double extension $A$ of $R$ has a DE-data $\{P,\sigma,\delta,\tau\}$.


\section{Semiclassical limits of double extensions}

We begin with constructing a class of right double extensions $A$ such that $A$ has a semiclassical limit. The following assumptions are modifications of  \cite[Notation 1.1]{Oh12} and \cite[Notation 2.1]{ChOh3}.
\begin{notn}\label{ASSUM}
Let $t$ be an indeterminate. We assume that a 5-tuple $(\Lambda,\Bbb F,R, A, t-1)$ satisfies the following conditions (1)-(6):

(1) The first entry $\Lambda$ is a nonempty subset of the set ${\bf k}\setminus\{0,1\}$.

(2) The second entry
$\Bbb F$ is a subring of the ring of regular functions on $\Lambda\cup\{1\}$ containing ${\bf k}[t,t^{-1}]$, that is,
\begin{equation}\label{REG}
{\bf k}[t,t^{-1}]\subset \Bbb F\subset\{f/g\in{\bf k}(t)| f,g\in{\bf k}[t]\text{ such that } g(1)\neq0,g(\lambda)\neq0\ \forall\lambda\in\Lambda\}.
\end{equation}

(3) The third entry $R$ is a (possibly noncommutative) $\Bbb F$-algebra.

(4) The fourth entry $A$ is  a right double extension of $R$ with two new variables $y_1, y_2$ and  DE-data $\{P,\sigma,\delta,\tau\}$,
such that
\begin{equation}\label{COE}
\begin{aligned}
P&=\{p_{11}(t),p_{12}(t)\}\text{ such that } p_{11}(t), p_{12}(t)-1\in (t-1)\Bbb F,\\
\tau&=\{\tau_1,\tau_2,\tau_0\}\subset (t-1)R,\\
 \sigma(a)-I(a)&=\left(\begin{matrix}\sigma_{11}(a)&\sigma_{12}(a)\\ \sigma_{21}(a)&\sigma_{22}(a)\end{matrix}\right)-
 \left(\begin{matrix}a&0\\ 0&a\end{matrix}\right)\in(t-1)M_{2\times2}(R),\\
 \delta(a)&=\left(\begin{matrix}\delta_{1}(a)\\ \delta_{2}(a)\end{matrix}\right)\in (t-1)M_{2\times1}(R)
\end{aligned}
\end{equation}
for all $a\in R$.

(5) The fifth  entry $t-1$ is a nonzero, nonunit and non-zero-divisor of $A$ such that the factor $R_1:=R/(t-1)R$ is commutative.

(6) For each $\lambda\in\Lambda\cup\{1\}$, set
$$R_\lambda:=R/(t-\lambda)R,\ \ \ A_\lambda:=A/(t-\lambda)A$$
and denote by $\gamma_\lambda:R \to R_\lambda$ and $\gamma^A_\lambda:A \to A_\lambda$ the canonical homomorphisms of ${\bf k}$-algebras. Assume that $A_\lambda$ is a free left $R_\lambda$-module with basis  $\{\gamma^A_\lambda(y_1)^{i}\gamma^A_\lambda(y_2)^{j}|i, j\geq0\}$ for each $\lambda\in\Lambda\cup\{1\}$. For convenience, we will still write
$$y_1 \text{ and $y_2$ for $\gamma^A_\lambda(y_1)$ and $\gamma^A_\lambda(y_2)$, respectively.}$$

\end{notn}


\begin{lem}\label{COMM}
Assume that $(\Lambda,{\Bbb F}, R, A, t-1)$ satisfies Assumption~\ref{ASSUM}.

(1)   For each $\lambda\in\Lambda\cup\{1\}$, $\gamma_\lambda(f(t))=f(\lambda)+(t-\lambda)R$ for all $f(t)\in\Bbb F$, $A_\lambda$ contains $R_\lambda$ as a subalgebra and the restriction map
 of $\gamma^A_\lambda$ to $R$ is equal to $\gamma_\lambda$. (We identify $\gamma_\lambda(f(t))$ with $f(\lambda)\in{\bf k}$.)

(2) $A_1=R_1[y_1,y_2]$, the commutative polynomial ring as a ring.
\end{lem}

\begin{proof}
(1)  Note that, for any $f(t)\in\Bbb F$, $f(\lambda)$ is a well-defined element of ${\bf k}$ by (\ref{REG}) and that $f(t)$ may be written by $f(t)=f(\lambda)+(t-\lambda)h(t)$ for some $h(t)\in\Bbb F$. Hence $\gamma_\lambda(f(t))=f(\lambda)+(t-\lambda)R$.

Let $a\in (t-\lambda)A\cap R$. Then $a=(t-\lambda)x$ for some $x\in A$ and $x$ is expressed uniquely as $x=\sum_{i,j}b_{ij}y_1^iy_2^j$, $b_{ij}\in R$, by  Assumption~\ref{ASSUM}(4). Thus  we have $a=\sum_{i,j}(t-\lambda)b_{ij}y_1^iy_2^j\in R$. It follows that
$a=(t-\lambda)b_{00}\in (t-\lambda)R$ by Assumption~\ref{ASSUM}(4). Hence $(t-\lambda)A\cap R=(t-\lambda)R$ and thus
the canonical map
$$R_\lambda\longrightarrow A_\lambda,\ \ \ a+(t-\lambda)R\mapsto a+(t-\lambda)A$$ is injective. It follows that $A_\lambda$ contains $R_\lambda$ as a subalgebra and the restriction map
 of $\gamma^A_\lambda$ to $R$ is equal to $\gamma_\lambda$.

(2) By Assumption~\ref{ASSUM}(5), $R_1$ is commutative and $A_1$  contains $R_1$ as a subring by the above result (1). Moreover, by (\ref{REL1}), (\ref{REL2}) and  (\ref{COE}),
\begin{eqnarray}
y_2y_1-y_1y_2=p_{11}(t)y_1^2+(p_{12}(t)-1)y_1y_2+\tau_1y_1+\tau_2y_2+\tau_0\in (t-1)A,\label{COMM1}\\
y_1a-ay_1=(\sigma_{11}-I)(a)y_1+\sigma_{12}(a)y_2+\delta_1(a)\in(t-1)A,\label{COMM2}\\
y_2a-ay_2=\sigma_{21}(a)y_1+(\sigma_{22}-I)(a)y_2+\delta_2(a)\in(t-1)A\label{COMM3}
\end{eqnarray}
for all $a\in R$, where $I(a)=a$. Hence
$A_1$ is commutative and thus $A_1$ is the commutative ring $R_1[y_1,y_2]$ by Assumption~\ref{ASSUM}(6).
\end{proof}

By Assumption~\ref{ASSUM}(5) and Lemma~\ref{COMM}(2), $A_1$ is a Poisson {\bf k}-algebra and hence is a semiclassical limit of $A$.
For convenience, we will write $\overline{x}$ for $\gamma_1^A(x)=x+(t-1)A\in A_1$ and $\overline{a}$ for $\gamma_1(a)=a+(t-1)R\in R_1$.

\begin{thm}\label{DPOE}
Assume that $(\Lambda,{\Bbb F}, R, A, t-1)$ satisfies Assumption~\ref{ASSUM}.

(1) For each $\lambda\in\Lambda$, the deformation $A_\lambda$ of $A_1$ is a right double extension of $R_\lambda$ with two variables $y_1$, $y_2$ and DE-data $\{P_\lambda,\sigma_\lambda,\delta_\lambda,\tau_\lambda\}$, where
$$\begin{aligned}
P_\lambda&=\gamma_\lambda(P)=\{p_{11}(\lambda),p_{12}(\lambda)\}\subset {\bf k}, \\
\sigma_\lambda&=\gamma_\lambda\sigma\gamma_\lambda^{-1}=\left(\begin{matrix}\gamma_\lambda\sigma_{11}\gamma_\lambda^{-1} &\gamma_\lambda\sigma_{12}\gamma_\lambda^{-1}\\
\gamma_\lambda\sigma_{21}\gamma_\lambda^{-1} &\gamma_\lambda\sigma_{22}\gamma_\lambda^{-1}\end{matrix}\right),\\
\delta_\lambda&=\gamma_\lambda\delta\gamma_\lambda^{-1}=\left(\begin{matrix}\gamma_\lambda\delta_1\gamma_\lambda^{-1}\\ \gamma_\lambda\delta_2\gamma_\lambda^{-1}\end{matrix}\right),\\
\tau_\lambda&=\gamma_\lambda(\tau)=\{\gamma_\lambda(\tau_1), \gamma_\lambda(\tau_2), \gamma_\lambda(\tau_0)\}\subset R_\lambda.
\end{aligned}$$

(2) The semiclassical limit $A_1=R_1[y_1,y_2]$ is a Poisson ${\bf k}$-algebra containing $R_1$ as a Poisson subalgebra with Poisson bracket
\begin{equation}\label{PB}
\begin{aligned}
\{y_2,y_1\}&=q_{11}y_1^2+q_{12}y_1y_2+w_1y_1+w_2y_2+w_0,\\
\{y_1,a\}&=\alpha_{11}(a)y_1+\alpha_{12}(a)y_2+\nu_1(a),\\
\{y_2,a\}&=\alpha_{21}(a)y_1+\alpha_{22}(a)y_2+\nu_2(a)
\end{aligned}
\end{equation}
for $a\in R_1$,
where
\begin{equation}\label{PDE1}
\begin{aligned}
q_{11}&=\overline{(t-1)^{-1}p_{11}(t)},&\  q_{12}&=\overline{(t-1)^{-1}(p_{12}(t)-1)},&&\\
w_1&=\overline{(t-1)^{-1}\tau_1},&\  w_2&=\overline{(t-1)^{-1}\tau_2},&\  w_0&=\overline{(t-1)^{-1}\tau_0},
\end{aligned}
\end{equation}
\begin{equation}\label{PDE2}
\begin{aligned}
\alpha_{11}&=\gamma_1[(t-1)^{-1}(\sigma_{11}-I)]\gamma_1^{-1}, & \alpha_{12}&=\gamma_1[(t-1)^{-1}\sigma_{12}]\gamma_1^{-1},\\
\alpha_{21}&=\gamma_1[(t-1)^{-1}\sigma_{21}]\gamma_1^{-1}, & \alpha_{22}&=\gamma_1[(t-1)^{-1}(\sigma_{22}-I)]\gamma_1^{-1},
\end{aligned}
\end{equation}
\begin{equation}\label{PDE3}
\begin{aligned}
\nu_1&=\gamma_1[(t-1)^{-1}\delta_1]\gamma_1^{-1}, & \nu_2&=\gamma_1[(t-1)^{-1}\delta_2]\gamma_1^{-1}.
\end{aligned}
\end{equation}
\end{thm}

$$
\xymatrix{
  & &R\subset A \ar@{~>}[dll] \ar[drr]^-{\txt{\qquad semiclassical limit}}\\
 R_\lambda \subset A_\lambda \ar@{<--}[rrrr]^(.45){\txt{deformation}}& & & & R_1\subset A_1=R_1[y_1,y_2]  }
$$
\begin{proof}
(1) Note that all maps $\gamma_\lambda\sigma_{ij}\gamma_\lambda^{-1}$ and $\gamma_\lambda\delta_{i}\gamma_\lambda^{-1}$ ($i,j=1,2$) are well-defined since $\sigma_{ij},\delta_i$ are  $\Bbb F$-linear maps.
The result follows by Assumption~\ref{ASSUM}(4), (6) and Lemma~\ref{COMM}(1).

(2)
Note that $q_{11}, q_{12}$, $w_1, w_2, w_0$, $\alpha_{ij}$ and $\nu_i$ $(i,j=1,2$) are well-defined by (\ref{COE}). The semiclassical limit $A_1$ of $A$ contains  $R_1$ as a Poisson subalgebra by Lemma~\ref{COMM}(2) and it is checked clearly by (\ref{COMM1}), (\ref{COMM2}), (\ref{COMM3}) that the Poisson bracket of $A_1$ satisfies (\ref{PB}).
\end{proof}

\begin{rem}
Here we reformulate the semiclassical limit of $A$ in the case that $A$ is a {\it left} double extension of $R$.

(1) In Assumption~\ref{ASSUM}, let $A$ be a {\it left} double extension of $R$ with DE-data $\{P',\sigma',\delta', \tau'\}$. Then, in order for $A_1$ to be commutative, we assume that  the DE-data $\{P',\sigma',\delta', \tau'\}$ satisfies  the same conditions in (\ref{COE})  as follows:
\begin{equation}\label{COE1}
\begin{aligned}
P'&=\{p'_{11}(t),p'_{12}(t)\}\text{ such that } p'_{11}(t), p'_{12}(t)-1\in (t-1)\Bbb F,\\
\tau'&=\{\tau'_1,\tau'_2,\tau'_0\}\subset (t-1)R,\\
 \sigma'(a)-I(a)&=\left(\begin{matrix}\sigma'_{11}(a)&\sigma'_{12}(a)\\ \sigma'_{21}(a)&\sigma'_{22}(a)\end{matrix}\right)-
 \left(\begin{matrix}a&0\\ 0&a\end{matrix}\right)\in(t-1)M_{2\times2}(R),\\
 \delta'(a)&=\left(\begin{matrix}\delta'_{1}(a)& \delta'_{2}(a)\end{matrix}\right)\in (t-1)M_{1\times2}(R)
\end{aligned}
\end{equation}
for all $a\in R$.

(2)  In the case that $A$ is a {\it left} double extension of $R$ with DE-data $\{P',\sigma',\delta', \tau'\}$ satisfying (\ref{COE1}), the  Poisson bracket of the semiclassical limit $A_1=R_1[y_1,y_2]$  is given as follows:
\begin{equation}\label{PBL11}
\begin{aligned}
\{y_2,y_1\}&=-(q'_{11}y_1^2+q'_{12}y_1y_2+w'_1y_1+w'_2y_2+w'_0),\\
\{y_1,a\}&=-[\alpha'_{11}(a)y_1+\alpha'_{21}(a)y_2+\nu'_1(a)],\\
\{y_2,a\}&=-[\alpha'_{12}(a)y_1+\alpha'_{22}(a)y_2+\nu'_2(a)]
\end{aligned}
\end{equation}
for $a\in R_1$,
where
$$
\begin{aligned}
q'_{11}&=\overline{(t-1)^{-1}p'_{11}(t)},&\  q'_{12}&=\overline{(t-1)^{-1}(p'_{12}(t)-1)},&&\\
w'_1&=\overline{(t-1)^{-1}\tau'_1},&\  w'_2&=\overline{(t-1)^{-1}\tau'_2},&\  w'_0&=\overline{(t-1)^{-1}\tau'_0},
\end{aligned}
$$
$$
\begin{aligned}
\alpha'_{11}&=\gamma_1[(t-1)^{-1}(\sigma'_{11}-I)]\gamma_1^{-1}, & \alpha'_{12}&=\gamma_1[(t-1)^{-1}\sigma'_{12}]\gamma_1^{-1},\\
\alpha'_{21}&=\gamma_1[(t-1)^{-1}\sigma'_{21}]\gamma_1^{-1}, & \alpha'_{22}&=\gamma_1[(t-1)^{-1}(\sigma'_{22}-I)]\gamma_1^{-1},
\end{aligned}
$$
$$
\begin{aligned}
\nu'_1&=\gamma_1[(t-1)^{-1}\delta'_1]\gamma_1^{-1}, & \nu'_2&=\gamma_1[(t-1)^{-1}\delta'_2]\gamma_1^{-1}.
\end{aligned}
$$

(3) Suppose that $A$ is a  double extension of $R$ with {\it right} DE-data $\{P,\sigma,\delta, \tau\}$ and {\it left} DE-data $\{P',\sigma',\delta', \tau'\}$  respectively. Then the {right} DE-data $\{P,\sigma,\delta, \tau\}$ satisfies (\ref{COE}) if and only if the { left} DE-data $\{P',\sigma',\delta', \tau'\}$ satisfies (\ref{COE1}). Moreover, if $A$ satisfies (\ref{COE}) or (\ref{COE1}) then the semiclassical limit induced by the right DE-data $\{P,\sigma,\delta, \tau\}$ is equal to that induced by the left DE-data  $\{P',\sigma',\delta', \tau'\}$. In fact, in (\ref{PB}) and (\ref{PBL11}),
\begin{equation}\label{relation-l-r}
q_{1i}=-q'_{1i},  \alpha_{ij}+\alpha'_{ji}=0, \nu_i=-\nu'_i, w_i=-w'_i,
\end{equation}
for all $i, j\in\{1,2\}$. By equation~\eqref{relation-l-r}, one can also obtain the corresponding version of the upcoming Lemma \ref{map} and Proposition    \ref{Poisson structure}  for {\it PDE-data}  $\{q',\alpha',\nu',w'\}$,   which is a semiclassical limit of left double extension.
\end{rem}

\begin{proof}
(2) Observe that
\begin{eqnarray}
y_2y_1-y_1y_2=-[p'_{11}(t)y_1^2+(p'_{12}(t)-1)y_2y_1+y_1\tau'_1+y_2\tau'_2+\tau'_0]\in (t-1)A,\label{COMM4}\\
y_1a-ay_1=-[y_1(\sigma'_{11}-I)(a)+y_2\sigma'_{21}(a)+\delta'_1(a)]\in(t-1)A,\label{COMM5}\\
y_2a-ay_2=-[y_1\sigma'_{12}(a)+y_2(\sigma'_{22}-I)(a)+\delta'_2(a)]\in(t-1)A\label{COMM6}
\end{eqnarray}
for all $a\in R$, where $I(a)=a$. Hence
$A_1$ is commutative, and $A_1$ is the commutative ring $R_1[y_1,y_2]$ by (\ref{COE1}).
Moreover the Poisson bracket (\ref{PBL11}) in $A_1$ is obtained  by (\ref{COMM4}),  (\ref{COMM5}), (\ref{COMM6})
since $A_1$ is commutative.

(3)
If the {right} DE-data $\{P,\sigma,\delta, \tau\}$ satisfies (\ref{COE}) then $A_1$ is commutative. It follows that
the { left} DE-data $\{P',\sigma',\delta', \tau'\}$ satisfies (\ref{COE1}) by
 (\ref{COMM4}),  (\ref{COMM5}), (\ref{COMM6}), since $A$ is a free right $R$-module with basis $\{y_2^iy_1^j| i,j\ge0\}$. The converse is proved similarly. The remaining results follow immediately by (\ref{PB}) and (\ref{PBL11}).
\end{proof}


Retain the notations in Theorem~\ref{DPOE}. Note that $R_1$ and $A_1=R_1[y_1,y_2]$ are Poisson algebras. Set
 $$\begin{aligned}
\left\{a,\left(\begin{matrix}b_{11}&b_{12}\\ b_{21}&b_{22}\end{matrix}\right)\right\}&:=
 \left(\begin{matrix}\{a, b_{11}\}&\{a, b_{12}\}\\ \{a, b_{21}\}&\{a,b_{22}\}\end{matrix}\right), &\left\{\left(\begin{matrix}b_{11}&b_{12}\\ b_{21}&b_{22}\end{matrix}\right),a\right\} &:=-\left\{a,\left(\begin{matrix}b_{11}&b_{12}\\ b_{21}&b_{22}\end{matrix}\right)\right\}\\
\left\{a,\left(\begin{matrix} z_1\\ z_2\end{matrix}\right)\right\}&:=
\left(\begin{matrix} \{a,z_1\}\\ \{a,z_2\}\end{matrix}\right), &\left\{\left(\begin{matrix} z_1\\ z_2\end{matrix}\right),a\right\}&:=-\left\{a,\left(\begin{matrix} z_1\\ z_2\end{matrix}\right)\right\}\\
\end{aligned}$$
for $a\in R_1$, $\left(\begin{matrix}b_{11}&b_{12}\\ b_{21}&b_{22}\end{matrix}\right)\in M_{2\times2}(R_1)$ and $\left(\begin{matrix} z_1\\ z_2\end{matrix}\right)\in M_{2\times1}(R_1[y_1,y_2])$.
  Define  ${\bf k}$-linear maps $\alpha$, $\nu$ by
$$\begin{aligned}
&\alpha:R_1\longrightarrow M_{2\times2}(R_1),&\ \ \ \alpha(a)&:=\left(\begin{matrix}\alpha_{11}(a)&\alpha_{12}(a)\\ \alpha_{21}(a)&\alpha_{22}(a)\end{matrix}\right),\\
&\nu:R_1\longrightarrow M_{2\times1}(R_1),&\ \ \ \nu(a)&:=\left(\begin{matrix}\nu_1(a)\\ \nu_2(a)\end{matrix}\right)
\end{aligned}$$ and set $y:=\left(\begin{matrix} y_1\\ y_2\end{matrix}\right)\in M_{2\times1}(R_1[y_1,y_2])$. Then, by (\ref{PB}),
\begin{equation}\label{PBB11}
\begin{aligned}
\{y,a\}&=\left(\begin{matrix} \{y_1,a\}\\ \{y_2,a\}\end{matrix}\right)=\left(\begin{matrix} \alpha_{11}(a)y_1+\alpha_{12}(a)y_2+\nu_1(a)\\ \alpha_{21}(a)y_1+\alpha_{22}(a)y_2+\nu_2(a)\end{matrix}\right)=\alpha(a)y+\nu(a)
\end{aligned}
\end{equation}
for all $a\in R_1$. Set
$$q=\{q_{11}, q_{12}\},\ \ \ w=\{w_1,w_2,w_0\},$$
where $q_{11}$, $q_{12}$, $w_1$, $w_2$, $w_0$ are those given in (\ref{PDE1}). We will call  $\{q,\alpha,\nu,w\}$  a {\it PDE-data} of $A_1=R_1[y_1,y_2]$ as in the double extension.

\begin{lem}\label{map}
Retain the above notations.
Then,  for any $a,b\in R_1$,

(1) $\alpha(ab)=a\alpha(b)+b\alpha(a)$.

(2) $\nu(ab)=a\nu(b)+b\nu(a)$.

(3) $\alpha(\{a,b\})=\{\alpha(a),b\}+\{a,\alpha(b)\}+[\alpha(a),\alpha(b)]$.

(4) $\nu(\{a,b\})=\{\nu(a),b\}+\{a,\nu(b)\}+\alpha(a)\nu(b)-\alpha(b)\nu(a)$.
\end{lem}

\begin{proof}
 Note that $\{y,ab\}=a\{y,b\}+b\{y,a\}$. By (\ref{PBB11}), we have that
$$\{y,ab\}=\alpha(ab)y+\nu(ab)$$
and
\begin{align*}
\{y,ab\}&=a\{y,b\}+b\{y,a\}\\
&=(a\alpha(b)+b\alpha(a))y+a\nu(b)+b\nu(a).
\end{align*}
Hence (1) and (2) follow by comparing the coefficients.

 Note, by (\ref{PBB11}), that
$$\{\alpha(a)y,b\}=\{\alpha(a),b\}y+\alpha(a)\{y,b\}=(\{\alpha(a),b\}+\alpha(a)\alpha(b))y+\alpha(a)\nu(b)$$ for all $a,b\in R_1$. We have
\begin{align*}
\{y,\{a,b\}\}&=\alpha(\{a,b\})y+\nu(\{a,b\})
\end{align*}
and
\begin{align*}
\{y,\{a,b\}\}&=\{\{y,a\},b\}-\{\{y,b\},a\}\\
&=\{\alpha(a)y+\nu(a),b\}-\{\alpha(b)y+\nu(b),a\}\\
&=(\{\alpha(a),b\}+\{a,\alpha(b)\}+[\alpha(a),\alpha(b)])y\\
&\qquad\qquad+\{\nu(a),b\}+\{a,\nu(b)\}+\alpha(a)\nu(b)-\alpha(b)\nu(a),
\end{align*}
by (\ref{PBB11}) and Jacobi identity of $R_1[y_1,y_2]$.
Hence (3) and (4) follow by comparing the coefficients.
\end{proof}

\begin{lem}\label{map1}
Retain the notations  of Lemma~\ref{map}. Then
$\{y,\{a,b\}\}+\{a,\{b,y\}\}+\{b,\{y,a\}\}=0$ for all $a,b\in R_1$
if and only if the formulas (3) and (4) of Lemma~\ref{map} hold.
\end{lem}

\begin{proof}
This has been proved in the proof of Lemma~\ref{map}(3) and (4) already.
\end{proof}

Let us summarize  some properties of the PDE-data $\{q,\alpha,\nu,w\}$ of $A_1=R_1[y_1,y_2]$.

\begin{prop}\label{Poisson structure}
The PDE-data $\{q,\alpha,\nu,w\}$ of $A_1=R_1[y_1,y_2]$ satisfies the following conditions (1)-(13). For all $a, b\in R_1$,
\begin{enumerate}
\item  $\alpha_{11},\alpha_{12},\alpha_{21},\alpha_{22}, \nu_1,\nu_2$ are all derivations.
\item $\alpha_{11} ( \{a,b\}) -\{ \alpha_{11}(a),b\} -\{a,\alpha_{11}(b)\}= \alpha_{12}(a) \alpha_{21}(b)-\alpha_{21}(a) \alpha_{12}(b)$.
\item $\alpha_{22} ( \{a,b\}) -\{ \alpha_{22}(a),b\} -\{a,\alpha_{22}(b)\}= \alpha_{21}(a) \alpha_{12}(b)-\alpha_{12}(a) \alpha_{21}(b)$.
\item $\alpha_{12} ( \{a,b\}) -\{ \alpha_{12}(a),b\} -\{a,\alpha_{12}(b)\}= \alpha_{11}(a) \alpha_{12}(b)-\alpha_{12}(a) \alpha_{11}(b) +  \alpha_{12}(a) \alpha_{22}(b)-\alpha_{22}(a) \alpha_{12}(b)$.
\item $\alpha_{21} ( \{a,b\}) -\{ \alpha_{21}(a),b\} -\{a,\alpha_{21}(b)\}= \alpha_{21}(a) \alpha_{11}(b)-\alpha_{11}(a) \alpha_{21}(b) +  \alpha_{22}(a) \alpha_{21}(b)-\alpha_{21}(a) \alpha_{22}(b)$.
\item $\nu_1 ( \{a,b\}) -\{ \nu_1(a),b\} -\{a,\nu_1(b)\}= \alpha_{11}(a) \nu_1(b)-\nu_1(a) \alpha_{11}(b)+\alpha_{12}(a) \nu_2(b)-\nu_2(a) \alpha_{12}(b)$.
\item $\nu_2 ( \{a,b\}) -\{ \nu_2(a),b\} -\{a,\nu_2(b)\}= \alpha_{21}(a) \nu_1(b)-\nu_1(a) \alpha_{21}(b)+\alpha_{22}(a) \nu_2(b)-\nu_2(a) \alpha_{22}(b)$.
\item $[\alpha_{21},\alpha_{11}]=q_{11}\alpha_{11}+q_{12}\alpha_{21}-q_{11}\alpha_{22}$.
\item $[\alpha_{22},\alpha_{11}]+[\alpha_{21},\alpha_{12}]=2q_{11}\alpha_{12}$.
\item $[\alpha_{22},\alpha_{12}]=q_{12}\alpha_{12}$.
\item $[\nu_2,\alpha_{11}]+[\alpha_{21},\nu_1]=2q_{11}\nu_1+q_{12}\nu_2+w_2 \alpha_{21}-w_1 \alpha_{22}+\{w_1,-\}$.
\item $[\nu_2,\alpha_{12}]+[\alpha_{22},\nu_1]=q_{12}\nu_1 +w_1 \alpha_{12}-w_2 \alpha_{11} + \{w_2,-\}$.
\item $[\nu_2,\nu_1]=w_1 \nu_1 +w_2 \nu_2-w_0\alpha_{11}-w_0 \alpha_{22}+\{w_0,-\}$.
\end{enumerate}
\end{prop}

\begin{proof}

(1) Note that  $\alpha_{11},\alpha_{12},\alpha_{21},\alpha_{22}$ are all derivations if and only if $\alpha(ab)=a\alpha(b)+b\alpha(a)$ for all $a,b\in R_1$
and that $\nu_1$ and $\nu_2$ are derivations if and only if $\nu(ab)=a\nu(b)+b\nu(a)$ for all $a,b\in R_1$. Hence it follows by Lemma~\ref{map}(1) and (2).

Conditions (2)-(5) follow immediately from  the  fact that they are equivalent to  Lemma~\ref{map}(3).

Conditions (6)-(7) follow immediately from the fact that they are equivalent to  Lemma~\ref{map}(4).

Finally, conditions (8)-(13) are proved by the fact that they are  equivalent to the Jacobi identities for $y_1$, $y_2$ and all  $a\in R_1$, $$\{y_2,\{y_1,a\}\}+\{y_1,\{a,y_2\}\}+\{a,\{y_2,y_1\}\}=0,$$
as follows.
Observe the following:
\begin{align*}
&\{\{y_2,y_1\},a\}\\
=& \{q_{11}y_1^2,a\}+\{q_{12}y_1y_2,a\}+\{w_1y_1,a\}+\{w_2y_2,a\}+\{w_0,a\}\\
=&(2q_{11}y_1+q_{12}y_2+w_1)\{y_1,a\}+(q_{12}y_1+w_2)\{y_2,a\}+y_1\{w_1,a\}+y_2\{w_2,a\}+\{w_0,a\}\\
=&(2q_{11}\alpha_{11}(a)+q_{12}\alpha_{21}(a))y_1^2+q_{12}\alpha_{12}(a)y_2^2+(2q_{11}\alpha_{12}(a)+q_{12}\alpha_{22}(a)+q_{12}\alpha_{11}(a))y_1y_2\\
&+(2q_{11}\nu_1(a)+q_{12}\nu_2(a)+w_1 \alpha_{11}(a)+w_2\alpha_{21}(a)+\{w_1,a\})y_1\\
&+(q_{12}\nu_1(a)+w_1 \alpha_{12}(a)+w_2\alpha_{22}(a)+\{w_2,a\})y_2+w_1\nu_1(a)+w_2\nu_2(a)+\{w_0,a\},
\end{align*}
\begin{align*}
&\{y_2,\{y_1,a\}\}\\
=&\{y_2,\alpha_{11}(a)y_1\}+\{y_2,\alpha_{12}(a)y_2\}+\{y_2,\nu_1(a)\}\\
=&\alpha_{11}(a)\{y_2,y_1\}+y_1\{y_2,\alpha_{11}(a)\}+y_2\{y_2,\alpha_{12}(a)\}+\{y_2,\nu_1(a)\}\\
=&(q_{11}\alpha_{11}(a)+\alpha_{21}\alpha_{11}(a))y_1^2+\alpha_{22}\alpha_{12}(a)y_2^2+(q_{12}\alpha_{11}(a)+\alpha_{22}\alpha_{11}(a)+\alpha_{21}\alpha_{12}(a))y_1y_2\\
&+(w_1\alpha_{11}(a)+\nu_2 \alpha_{11}(a) +\alpha_{21}\nu_1(a))y_1+(w_2\alpha_{11}(a)+\nu_2 \alpha_{12}(a)+\alpha_{22}\nu_1(a))y_2\\
&+w_0 \alpha_{11}(a)+\nu_2 \nu_1(a),
\end{align*}
\begin{align*}
&\{y_1,\{y_2,a\}\}\\
=&\{y_1,\alpha_{21}(a)y_1\}+\{y_1,\alpha_{22}(a)y_2\}+\{y_1,\nu_2(a)\}\\
=&y_1\{y_1,\alpha_{21}(a)\}+\alpha_{22}(a)\{y_1,y_2\}+y_2\{y_1,\alpha_{22}(a)\}+\{y_1,\nu_2(a)\}\\
=&(\alpha_{11}\alpha_{21}(a)-q_{11}\alpha_{22}(a))y_1^2+\alpha_{12}\alpha_{22}(a)y_2^2+(\alpha_{11}\alpha_{22}(a)+\alpha_{12}\alpha_{21}(a)-q_{12}\alpha_{22}(a))y_1y_2\\
&+(\nu_1 \alpha_{21}(a) +\alpha_{11}\nu_2(a)-w_1\alpha_{22}(a))y_1+(\nu_1 \alpha_{22}(a)+\alpha_{12}\nu_2(a)-w_2\alpha_{22}(a))y_2\\
&+\nu_1 \nu_2(a)-w_0 \alpha_{22}(a).
\end{align*}
Comparing coefficients of  $y_1^2, y_2^2, y_1y_2, y_1, y_2$ and the constant terms respectively, we have that $\{y_2,\{y_1,a\}\}=\{\{y_2,y_1\},a\}+\{y_1,\{y_2,a\}\}$ for all $a\in R_1$ if and only if the conditions (8)-(13) hold.
\end{proof}
\begin{rem}
 The conditions (1)-(13) in the above proposition can be thought as consequences of taking semiclassical limits of those conditions in \cite[Lemma 1.7(a), (b) and (R3)]{ZhZh} for double extensions.
\end{rem}
Now we can characterize the fact that the polynomial ring $R_1[y_1,y_2]$ is a Poisson algebra with Poisson bracket (\ref{PB}).

\begin{thm}\label{LOW1}
 Let $R$ be a Poisson ${\bf k}$-algebra with Poisson bracket $\{-,-\}_R$ and  $R[y_1,y_2]$ be the commutative polynomial ring over $R$. Set
 $$\begin{aligned}
&q=\{q_{11},q_{12}\}\subset {\bf k},&w&=\{w_1,w_2,w_0\}\subset R,\\
&\alpha:R\longrightarrow M_{2\times2}(R),&\ \ \ \alpha(a)&=\left(\begin{matrix}\alpha_{11}(a)&\alpha_{12}(a)\\ \alpha_{21}(a)&\alpha_{22}(a)\end{matrix}\right),\\
&\nu:R\longrightarrow M_{2\times1}(R),&\ \ \ \nu(a)&=\left(\begin{matrix}\nu_{1}(a)\\ \nu_{2}(a)\end{matrix}\right),\\
&y=\left(\begin{matrix} y_1\\ y_2\end{matrix}\right)\in M_{2\times1}(R[y_1,y_2]).
\end{aligned}$$
 Then $R[y_1,y_2]$ becomes a Poisson algebra with Poisson bracket $\{-,-\}$ such that
\begin{equation}\label{PB1}
\begin{aligned}
\{a,b\}&=\{a,b\}_R,\\
\{y_2,y_1\}&=q_{11}y_1^2+q_{12}y_1y_2+w_1y_1+w_2y_2+w_0,\\
\{y_1,a\}&=\alpha_{11}(a)y_1+\alpha_{12}(a)y_2+\nu_1(a),\\
\{y_2,a\}&=\alpha_{21}(a)y_1+\alpha_{22}(a)y_2+\nu_2(a)
\end{aligned}
\end{equation}
for all $a,b \in R$ if and only if the PDE-data $\{q,\alpha,\nu,w\}$ satisfies the conditions (1)-(13) of Proposition~\ref{Poisson structure},  where $R_1$ is replaced by $R$.

Shortly speaking, $R[y_1,y_2]$ becomes a Poisson algebra with Poisson bracket
\begin{equation}\label{PB11}
\begin{aligned}
\{a,b\}&=\{a,b\}_R,\\
\{y_2,y_1\}&=q_{11}y_1^2+q_{12}y_1y_2+w_1y_1+w_2y_2+w_0,\\
\{y,a\}&=\alpha(a)y+\nu(a)
\end{aligned}
\end{equation}
for all $a,b\in R$ if and only if the PDE-data $\{q,\alpha,\nu,w\}$ satisfies the following conditions (a)-(e) for all $a,b\in R$.

(a)  $\alpha(ab)=a\alpha(b)+b\alpha(a)$.

(b) $\nu(ab)=a\nu(b)+b\nu(a)$.

(c) $\alpha(\{a,b\})=\{\alpha(a),b\}+\{a,\alpha(b)\}+[\alpha(a),\alpha(b)]$.

(d) $\nu(\{a,b\})=\{\nu(a),b\}+\{a,\nu(b)\}+\alpha(a)\nu(b)-\alpha(b)\nu(a)$.

(e) $\{y_2,\{y_1,a\}\}+\{y_1,\{a,y_2\}\}+\{a,\{y_2,y_1\}\}=0.$
\end{thm}

We will call the Poisson algebra $R[y_1,y_2]$ with Poisson bracket (\ref{PB1}) a {\it Poisson double extension} of $R$ with variables $y_1, y_2$ and  PDE-data $\{q,\alpha,\nu,w\}$. Moreover, a Poisson double extension is called {\it trimmed} if $\nu=w=0$ in the PDE-data. By \eqref{PDE1} and \eqref{PDE3}, it is easy to see that the semiclassical limits of trimmed double extensions yield trimmed Poisson double extensions.

\begin{proof}
Note that (\ref{PB1}) is equivalent to (\ref{PB11}) and that the conditions (1)-(13) of Proposition~~\ref{Poisson structure} are equivalent to (a)-(e) by the proof of Proposition~~\ref{Poisson structure}.

$(\Rightarrow)$
Suppose that $R[y_1,y_2]$ is a Poisson algebra with the Poisson bracket \eqref{PB1}. Then we obtain the conditions (1)-(13) of Proposition~\ref{Poisson structure}
in which $R_1$ is replaced by $R$.

$(\Leftarrow)$
Suppose that the PDE-data $\{q,\alpha,\nu,w\}$ satisfies the conditions (1)-(13) of Proposition~\ref{Poisson structure}. Extend derivations  $\{-,-\}_R$, $\alpha_{ij}$, $\nu_j$ ($i,j=1,2)$ of $R$ to $R[y_1,y_2]$
by
$$\{ay_1^ky_2^\ell, by_1^my_2^n\}_R=\{a,b\}_Ry_1^{k+m}y_2^{\ell+n},\ \ \alpha_{ij}(by_1^my_2^n)=\alpha_{ij}(b)y_1^my_2^n, \ \ \nu_{j}(by_1^my_2^n)=\nu_{j}(b)y_1^my_2^n$$
for all $k,\ell,m,n\geq0$, and $a, b\in R$.
Then the extensions  $\{-,-\}_R$, $\alpha_{ij}$, $\nu_j$ are derivations of $R[y_1,y_2]$ and thus
$$\begin{aligned}
\{-,-\}:&=\{-,-\}_R+\frac{\partial}{\partial y_1}\wedge (y_1\alpha_{11}+y_2\alpha_{12}+\nu_1)+\frac{\partial}{\partial y_2}\wedge (y_1\alpha_{21}+y_2\alpha_{22}+\nu_2)\\
&\qquad-(q_{11}y_1^2+q_{12}y_1y_2+\omega_1y_1+\omega_2y_2+\omega_0)\frac{\partial}{\partial y_1}\wedge\frac{\partial}{\partial y_2}
\end{aligned}$$
is skew-symmetric and satisfies Leibniz rule on $R[y_1,y_2]$. Moreover $\{-,-\}$ satisfies (\ref{PB1}). Hence it suffices to show that $\{-,-\}$ satisfies Jacobi identity.

 We see that the conditions (2)-(7) are equivalent to the conditions
 $$\{\{a,b\},y_i\}+\{\{b,y_i\},a\}+\{\{y_i,a\},b\}=0$$
 for $i=1,2$ by Lemma~\ref{map1} and that the conditions (8)-(13) are equivalent to the condition
 $$\{\{y_1,y_2\},a\}+\{\{y_2,a\},y_1\}+\{\{a,y_1\},y_2\}=0$$ by the proof of Proposition~\ref{Poisson structure}.
For polynomials $f,g,h\in R[y_1,y_2]$, suppose that $h=h_1h_2$ and that  the triples $(f,g,h_1)$ and $(f,g,h_2)$  satisfy Jacobi identity. Then
we have that
\begin{align*}
\{\{f,g\},h\}&+\{\{g,h\},f\}+\{\{h,f\},g\}\\
=\ &\{\{f,g\},h_1h_2\}+\{\{g,h_1h_2\},f\}+\{\{h_1h_2,f\},g\}\\
=\ &h_1 \{\{f,g\},h_2\}+ h_2 \{\{f,g\},h_1\}\\
\ &+\{h_1\{g,h_2\},f\}+\{h_2\{g,h_1\},f\}+\{h_1\{h_2,f\},g\}+\{h_2\{h_1,f\},g\}\\
=\ &h_1 ( \{\{f,g\},h_2\}+\{\{g,h_2\},f\}+\{\{h_2,f\},g\} )\\
\ &+h_2( \{\{f,g\},h_1\}+\{\{g,h_1\},f\}+\{\{h_1,f\},g\} )\\
\ &+\{h_1,f\}\{g,h_2\}+\{h_2,f\}\{g,h_1\}+\{h_1,g\}\{h_2,f\}+\{h_2,g\}\{h_1,f\}\\
=\ &0.
\end{align*}
Therefore Jacobi identity for $\{-,-\}$ holds by using induction on degrees of $f,g,h$.
\end{proof}

\begin{rem}
For a Poisson double extension $R[y_1,y_2]$ with  PDE-data $\{q,\alpha,\nu,w\}$,  it is enough to consider only the cases
$\{0,p\}$, $\{1,0\}$   and $\{0,0\}$ for $q=\{q_{11},q_{12}\}$, where $0\neq p\in{\bf k}$.
In fact, as in the case of double extension \cite[Remark~1.4]{ZhZh},    one can choose a suitable basis for the vector space ${\bf k}y_1+{\bf k}y_2$ to simplify the PDE-data $\{q, \alpha, \nu, w\}$ of the Poisson double extension $R[y_1,y_2]$ as follows:
  \begin{enumerate}
    \item If $q_{12}\neq 0$,  by setting $p=q_{11}q_{12}^{-1}, z_1=y_1,z_2=py_1+y_2$, then the Poisson double extension $R[y_1,y_2]$  can be presented by the Poisson double extension $R[z_1,z_2]$ with the variables $z_1,z_2$ and the  PDE-data $$\left\{
 \{0,q_{12}\},
  \left(\begin{matrix}\alpha_{11}-p\alpha_{12}&\alpha_{12}\\ \alpha_{21}+p(\alpha_{11}-\alpha_{22})-p^2\alpha_{12}&p\alpha_{12}+\alpha_{22}\end{matrix}\right),
   \left(\begin{matrix}\nu_1\\ \nu_2+p\nu_1\end{matrix}\right), \{w_1-pw_2,w_2,w_0\}\right\}.$$
    \item If $q_{12}=0$ and $q_{11}\neq 0$, by setting $z_1=q_{11}y_1,z_2=y_2$, then the Poisson double extension $R[y_1,y_2]$ can be presented by the Poisson double extension $R[z_1,z_2]$ with the variables $z_1,z_2$ and the  PDE-data
        $$\left\{
 \{1,0\},
  \left(\begin{matrix}\alpha_{11}&q_{11}\alpha_{12}\\ q_{11}^{-1}\alpha_{21}&\alpha_{22}\end{matrix}\right),
   \left(\begin{matrix}q_{11}\nu_1\\ \nu_2\end{matrix}\right), \{w_1, q_{11}w_2,q_{11}w_0\}\right\}.$$
 \item  If $q_{12}=0$ and $q_{11}= 0$,   then the Poisson double extension $R[y_1,y_2]$  has the  PDE-data
  $$\left\{\{0,0\}, \alpha, \nu, w\right\}.$$
  \end{enumerate}
\end{rem}

\section{Smoothness of Poisson double extensions and modular derivations}
In this section, we will prove that  Poisson double extensions preserve the smoothness of Poisson algebras, and give an explicit formula of modular derivations for such extensions.

Recall that the modular derivations are closely related to the Poisson homological duality theory for Poisson algebras.
In \cite{Kon}, Kontsevich investigated the relationship between the Poisson variety and its deformation quantizations, and related the Poisson (co)homologies of  Poisson algebras to the Hochschild (co)homologies of their deformation quantization algebras. After that, Dolgushev proved in \cite{Do} that: for a smooth Poisson affine variety $A$ with  trivial canonical bundle, its deformation quantization algebra $(A[[\hbar]], *)$ is Calabi-Yau if and only if the corresponding Poisson structure is unimodular, which is equivalent to say that the modular derivation is log-Hamiltonian. By using modular derivation, he also decoded the Poincar\'{e} duality theories of Hochschild (co)homologies and Poisson (co)homologies.

For a smooth Poisson algebra, a general Poincar\'{e} duality theorem was proved by Huebschmann in terms of Lie-Rinehart algebras \cite{Hu99}. See \cite{LWW} and \cite{Zhu14} for the twisted Poincar\'{e} duality theory of polynomial Poisson algebras and Frobenius Poisson algebras, respectively. All these results imply that the modular derivations of  Poisson algebras determine the concrete forms of duality theories. Hence, we will focus on the modular derivations of Poisson double extensions of Poisson algebras.

\begin{defn}\label{modular4} Let $R$ be a smooth Poisson algebra of
Krull dimension $n$ with trivial canonical bundle $\Omega^n(R)=R \vol$,
where $\vol$ is a volume form. The {\it modular derivation} of $R$
with respect to $\vol$ is defined as the map $\phi_{\vol}: R \to R$
such that for any $f\in R$,
$$\phi_{\vol}(f)=\frac{\mathscr{L}_{H_f}(\vol)}{\vol},$$ where  $H_f:=\{f,-\}: R \to R$ is the hamiltonian derivation associated to $f$,
$\mathscr{L}_{H_f}:=[\de, \iota_{H_f}]$ is the Lie-derivation on $\Omega^*(R)$,  and $\iota_{H_f}$ is the contraction map induced by the derivation $H_f$.
\end{defn}

\begin{rem}
 For more related materials on modular derivations, see \cite{LWW} and \cite{Wang}. In fact, the modular derivation is a Poisson derivation for Poisson algebra $R$. Moreover,  when the volume form is changed, the corresponding modular
derivation is modified by a so called log-hamiltonian derivation. That is,
if $\lambda$ is another volume form of $R$, then $\lambda = u \vol$ for some
unit $u \in R$, and $\phi_{\lambda} = \phi_{\vol} - u^{-1}H_u.$ The
Poisson derivation $ u^{-1}\{u, -\}: R \to R$ is called a {\it
log-hamiltonian derivation} of $R$ or an {\it inner Poisson derivation} determined by the invertible
element $u$.
\end{rem}

\begin{defn} For  an affine smooth Poisson algebra $R$ with trivial canonical bundle $R\vol$, $R$ is said to be {\it unimodular} if its modular derivation $\phi_{\vol}$ is a log-hamiltonian derivation. This is also equivalent to say one can change the volume form such that the corresponding modular derivation is zero.
\end{defn}

\begin{prop}
If $R$ is an affine smooth Poisson algebra with trivial canonical bundle, then so is its polynomial algebra $R[y_1,y_2]$.
\end{prop}
\begin{proof}
Let $S=R[y_1, y_2]$. Suppose that the dimension of $R$ is $n$ and that the canonical bundle $\Omega^n(R)\cong R\vol$. Since $S$ is a polynomial ring over $R$ with variables $y_1$ and $y_2$, $S\otimes_R(\Omega^1(R)\oplus R\de\! y_1\oplus R\de\!y_2)$ is the K\"{a}hler differential module for $S$.
It is  easy to see that $S$ is an affine smooth algebra of dimension $n+2$ and that its canonical bundle $\Omega^{n+2}(S)\cong S\vol\wedge\de\! y_1\wedge\de\! y_2$.
\end{proof}
\begin{thm}\label{main-result}
Let $R$ be an affine smooth Poisson algebra of dimension $n$ with trivial canonical bundle  $\Omega^n(R)\cong R\vol$ and $S=R[y_1, y_2]$ be its Poisson double extension with variables $y_1, y_2$ and  PDE-data $\{q,\alpha,\nu,w\}$. Suppose $\phi$ and $\psi$ are the modular derivations of $R$ and $S$ with respect to the volume forms $\vol$ and $\vol\wedge\de\! y_1\wedge\de\! y_2$, respectively. Then the restriction $\psi|_R$ of $\psi$ on  $R$  is $$\psi|_R=\phi-\alpha_{11}-\alpha_{22},$$ and
  \begin{align*}
  \psi(y_1) & =(u_{11}-q_{12})y_1+u_{12}y_2+v_1-w_2; \\
  \psi(y_2) & =(u_{21}+2q_{11})y_1+(u_{22}+q_{12})y_2+v_2+w_1,
\end{align*}
where the $u_{11}, u_{12}, u_{21}, u_{22}, v_1, v_2\in R$ are determined by $\mathscr{L}_{\alpha_{ij}}(\vol)=u_{ij}\vol$ and $\mathscr{L}_{\nu_i}(\vol)=v_i\vol$ for $i,j\in \{1,2\}$.
\end{thm}
\begin{proof}
By Definition~\ref{modular4}, for any $a\in R$,
\begin{align*}
\psi(a)\vol\wedge\de\! y_1\wedge\de\! y_2&=\mathscr{L}_{H_a}(\vol\wedge\de\! y_1\wedge\de\! y_2)\\
&=\de \circ\iota_{H_a}(\vol\wedge\de\! y_1\wedge\de\! y_2)\\
&=\de (\iota_{H_a}(\vol)\wedge\de\! y_1\wedge\de\! y_2+(-1)^n\{a,y_1\}\vol\wedge\de\! y_2+(-1)^{n+1}\{a,y_2\}\vol\wedge\de\! y_1)\\
&=\mathscr{L}_{H_a}(\vol)\wedge\de\! y_1\wedge\de\! y_2-(\alpha_{11}(a)+\alpha_{22}(a))\vol\wedge\de\! y_1\wedge\de\! y_2\\
&=(\phi(a)-\alpha_{11}(a)-\alpha_{22}(a))\vol\wedge\de\! y_1\wedge\de\! y_2.
\end{align*}
Moreover,
\begin{align*}
&\psi(y_1)\vol\wedge\de\! y_1\wedge\de\! y_2\\
=&\mathscr{L}_{H_{y_1}}(\vol\wedge\de\! y_1\wedge\de\! y_2)\\
=&\de \circ\iota_{H_{y_1}}(\vol\wedge\de\! y_1\wedge\de\! y_2)\\
=&\de (\iota_{H_{y_1}}(\vol\wedge\de\! y_1\wedge\de\! y_2)-(-1)^n\{y_1, y_2\}\vol\wedge\de\! y_1)\\
=&\de ((y_1\iota_{\alpha_{11}}(\vol)+y_2\iota_{\alpha_{12}}(\vol)+\iota_{\nu_{1}}(\vol))\wedge\de\! y_1\wedge\de\! y_2-(-1)^n\{y_1, y_2\}\vol\wedge\de\! y_1)\\
=&(y_1\mathscr{L}_{\alpha_{11}}(\vol)+y_2\mathscr{L}_{\alpha_{12}}(\vol)+ \mathscr{L}_{\nu_1}(\vol)) \wedge\de\! y_1\wedge\de\!y_2-(q_{12}y_1+w_2)\vol\wedge\de\! y_1\wedge\de\! y_2\\
=&((u_{11}-q_{12})y_1+u_{12}y_2+v_1-w_2)\vol\wedge\de\! y_1\wedge\de\! y_2,
\end{align*}
and
\begin{align*}
&\psi(y_2)\vol\wedge\de\! y_1\wedge\de\! y_2\\
=&\mathscr{L}_{H_{y_2}}(\vol\wedge\de\! y_1\wedge\de\! y_2)\\
=&\de \circ\iota_{H_{y_2}}(\vol\wedge\de\! y_1\wedge\de\! y_2)\\
=&\de (\iota_{H_{y_2}}(\vol\wedge\de\! y_1\wedge\de\! y_2)+(-1)^n\{y_2, y_1\}\vol\wedge\de\! y_2)\\
=&\de ((y_1\iota_{\alpha_{21}}(\vol)+y_2\iota_{\alpha_{22}}(\vol)+\iota_{\nu_{2}}(\vol))\wedge\de\! y_1\wedge\de\! y_2+(-1)^n\{y_2, y_1\}\vol\wedge\de\! y_2)\\
=&(y_1\mathscr{L}_{\alpha_{21}}(\vol)+y_2\mathscr{L}_{\alpha_{22}}(\vol)+ \mathscr{L}_{\nu_2}(\vol)) \wedge\de\! y_1\wedge\de\!y_2+(2q_{11}y_1+q_{12}y_2+w_1)\vol\wedge\de\! y_1\wedge\de\! y_2\\
=&((u_{21}+2q_{11})y_1+(u_{22}+q_{12})y_2+v_2+w_1)\vol\wedge\de\! y_1\wedge\de\! y_2.
\end{align*}
This completes the proof.
\end{proof}

\begin{cor}
Let $R$ and $S$ be as in Theorem~\ref{main-result}.  Then the trimmed Poisson double extension $S=R[y_1,y_2]$ with the following PDE-data is a unimodular Poisson algebra:
  $$\left\{
 \{0,0\},
  \left(\begin{matrix}\phi&0\\ 0&0\end{matrix}\right),
   \left(\begin{matrix}0\\0\end{matrix}\right), \{0,0,0\}\right\}.$$
\end{cor}
\begin{proof}
  It is easy to check that the equations in Proposition~\ref{Poisson structure} hold. By Theorem~\ref{main-result}, it suffices to check that the modular derivation of Poisson algebra $S$ is trivial.

  Since $q_{11}=q_{12}=0$, $\alpha_{12}=\alpha_{21}=\alpha_{22}=0$ and $\alpha_{11}=\phi$ is the modular derivation of $R$, $\psi(a)=\psi(y_2)=0$ for all $a\in R$. By \cite[Proposition~3.8]{Wang}, $\mathscr{L}_{\phi}(\vol)=0$, i.e., $\psi(y_1)=0$.
\end{proof}

\section{Relationship between Poisson double extensions and iterated Poisson polynomial extensions}

We recall the notion of Poisson polynomial extension, which is an analogue of Ore extension \cite{Oh8}. Let $R$ be a Poisson algebra. For   derivations $\beta$ and $\nu$ on $R$, the polynomial ring $R[x]$ becomes a Poisson algebra, containing $R$ as a Poisson subalgebra, with Poisson bracket $\{x,a\}=\beta(a)x+\nu(a)$ for all $a\in R$ if and only if $\beta$ is a Poisson derivation and the pair  $(\beta,\nu)$ satisfies the condition
\begin{equation}\label{SKEPOI}
\nu(\{a,b\})=\{\nu(a),b\}+\{a,\nu(b)\}+\beta(a)\nu(b)-\nu(a)\beta(b)
\end{equation}
for all $a,b\in R$. Such a Poisson algebra $R[x]$ is called a Poisson polynomial extension of $R$ and denoted by $R[x;\beta,\nu]_p$.

Here we investigate the relationship between Poisson double extensions and iterated Poisson polynomial extensions.

\begin{prop}\label{IteratedORE}
Let  $A=R[y_1,y_2]$ be a Poisson double extension of $R$ with variables $y_1,y_2$ and PDE-data
$$\left\{q=\{q_{11},q_{12}\},\ \alpha=\begin{pmatrix}\alpha_{11}&\alpha_{12}\\ \alpha_{21}&\alpha_{22}  \end{pmatrix},\ \nu=\begin{pmatrix}\nu_{1}\\ \nu_{2} \end{pmatrix},\ w=\{w_1,w_2,w_0\}\right\}.$$
Then $A$ can be presented by an iterated Poisson polynomial extension of the form
$$A=R[y_1;\beta_1,\delta_1]_p[y_2;\beta_2,\delta_2]_p\ \ \ \text{or}\ \ \ A=R[y_2;\beta_2,\delta_2]_p[y_1;\beta_1,\delta_1]_p$$
if and only if $\alpha_{12}=0$ or $\alpha_{21}=0, q_{11}=0$.
In such cases, the following hold.

(1) If $\alpha_{12}=0$ then $A=R[y_1;\alpha_{11},\nu_1]_p[y_2;\beta,\mu]_p,$ where $$\begin{aligned}
\beta(a)&=\alpha_{22}(a),&\beta(y_1)&=q_{12}y_1+w_2,\\
\mu(a)&=\alpha_{21}(a)y_1+\nu_2(a),&\mu(y_1)&=q_{11}y_1^2+w_1y_1+w_0
\end{aligned}$$
for all $a\in R$.

(2)  If $\alpha_{21}=0, q_{11}=0$ then $A=R[y_2;\alpha_{22},\nu_2]_p[y_1;\beta,\mu]_p,$ where $$\begin{aligned}
\beta(a)&=\alpha_{11}(a),&\beta(y_2)&=-q_{12}y_2-w_1,\\
\mu(a)&=\alpha_{12}(a)y_2+\nu_1(a),&\mu(y_2)&=-w_2y_2-w_0
\end{aligned}$$
for all $a\in R$.
\end{prop}

\begin{proof}
$(\Leftarrow)$ It is proved easily by (\ref{PB1}) and \cite[Theorem 1.1]{Oh8}.

$(\Rightarrow)$
Suppose that $A$ is of the form $A=R[y_1;\beta_1,\delta_1]_p[y_2;\beta_2,\delta_2]_p$. Then, for any $a\in R$,
$$\begin{aligned}
\{y_1,a\}&=\beta_1(a)y_1+\delta_1(a)\in Ry_1+R.
\end{aligned}$$
Since its left hand side is $\alpha_{11}(a)y_1+\alpha_{12}(a)y_2+\nu_1(a)$ by (\ref{PB1}), we have that $\alpha_{12}=0$.

Suppose that $A$ is of the form $A=R[y_2;\beta_2,\delta_2]_p[y_1;\beta_1,\delta_1]_p$. Then, for any $a\in R$,
$$\begin{aligned}
\{y_1,y_2\}&=\beta_1(y_2)y_1+\delta_1(y_2)\in R[y_2]y_1+R[y_2],\\
\{y_2,a\}&=\beta_2(a)y_2+\delta_2(a)\in Ry_2+R.
\end{aligned}$$
Hence
$$\begin{aligned}
-(q_{11}y_1^2+q_{12}y_1y_2+w_1y_1+w_2y_2+w_0)&\in R[y_2]y_1+R[y_2],\\
\alpha_{21}(a)y_1+\alpha_{22}(a)y_2+\nu_2(a)&\in Ry_2+R,
\end{aligned}$$
by (\ref{PB1}). Since $A$ is a free $R$-module with basis $\{y_1^iy_2^j| i,j\geq 0\}$,  $q_{11}=0$ and $\alpha_{21}=0$
\end{proof}

We will see in Example~\ref{NONIT} that there exists a Poisson double extension that is not an iterated Poisson polynomial extension.

\begin{lem}
Let $A$ be an iterated Poisson polynomial extension $A=R[y_1;\alpha_1,\nu_1]_p[y_2;\alpha_2,\nu_2]_p$ over  Poisson algebra $R$ such that
$\nu_2(R)\subset Ry_1+R.$
Then there exist ${\bf k}$-linear maps $\nu_2^1,\nu_2^0$ from $R$ to itself such that
\begin{equation}\label{SKEPOI2}
\nu_2(a)=\nu_2^1(a)y_1+\nu_2^0(a)
\end{equation} for all $a\in R.$
\end{lem}

\begin{proof}
Note that $A=R[y_1,y_2]$ is a free $R$-module with basis $\{y_1^iy_2^j|i,j\geq0\}$.
Since $\nu_2(R)\subset Ry_1+R$, there exist ${\bf k}$-linear maps $\nu_2^1,\nu_2^1$ from $R$ to itself such that
$$\nu_2(a)=\nu_2^1(a)y_1+\nu_2^0(a)$$
for each $a\in R$.
\end{proof}

\begin{prop}\label{itdoPe}
Let $A$ be an iterated Poisson polynomial extension $A=R[y_1;\alpha_1,\nu_1]_p[y_2;\alpha_2,\nu_2]_p$ over   Poisson algebra $R$ such that
\begin{itemize}
\item $\alpha_2(R)\subset R$,
\item $\nu_2(R)\subset Ry_1+R$,
\item $\alpha_2(y_1)=\mu_{12} y_1+w_2$ for some $\mu_{12}\in{\bf k}$ and $w_2\in R$,
\item  $\nu_2(y_1)=\mu_{11}y_1^2+w_1 y_1+w_0$ for some $\mu_{11}\in{\bf k}$ and $w_1, w_0\in R$.
\end{itemize}
Then
 $A$ is a Poisson double extension with variables $y_1, y_2$ and  PDE-data $\{q,\alpha,\nu,w\}$, where
$$q=\{\mu_{11},\mu_{12}\},\ \ \alpha(a)=\left(\begin{matrix} \alpha_1(a)&0\\ \nu_2^1(a)&\alpha_2(a)\end{matrix}\right),\ \
\nu(a)=\left(\begin{matrix} \nu_1(a)\\ \nu_2^0(a)\end{matrix}\right),\ \ w=\{w_1, w_2, w_0\}$$
and $\nu_2^1,\nu_2^0$ are given in (\ref{SKEPOI2}).
\end{prop}

\begin{proof}
Since $A=R[y_1,y_2]$ is a Poisson algebra with the following Poisson bracket
$$\begin{aligned}
\{a,b\}&=\{a,b\}_R,\\
\{y_2,y_1\}&=\alpha_2(y_1)y_2+\nu_2(y_1)=\mu_{12}y_1y_2+w_2y_2+\mu_{11}y_1^2+w_1 y_1+w_0,\\
\{y_1,a\}&=\alpha_1(a)y_1+\nu_1(a),\\
\{y_2,a\}&=\alpha_2(a)y_2+\nu_2(a)=\nu_2^1(a)y_1+\alpha_2(a)y_2+\nu_2^0(a),
\end{aligned}$$ for all $a,b\in R$,
the result follows from Theorem~\ref{LOW1}.
\end{proof}


\section{Examples}

In this section, we give examples of Poisson double extensions. In particular, we construct Poisson double extensions which are semiclassical limits of  the double extensions in \cite{ZhZh}.

Let us begin with  constructing a Poisson double extension $R[y_1,y_2]$ for any Poisson algebra $R$.

\begin{exam}\label{polynomial-alg}
For any Poisson algebras $A$ and $B$, note that $A\otimes_{\bf k} B$ is also a Poisson algebra with Poisson bracket
$$\{a_1\otimes b_1, a_2\otimes b_2\}=\{a_1,a_2\}\otimes b_1b_2+a_1a_2\otimes\{b_1,b_2\}$$
for all $a_1,a_2\in A$ and $b_1,b_2\in B$.

For any $q_{11},q_{12},w_1,w_2,w_0\in{\bf k}$, there exists a Poisson double extension ${\bf k}[y_1, y_2]$ of ${\bf k}$  with  PDE-data $\{q=\{q_{11},q_{12}\}, 0, 0, w=\{w_1,w_2,w_0\}\}$ 
since the PDE-data $\{q, 0, 0, w\}$ satisfies the conditions (a)-(e) of Theorem~\ref{LOW1}.
Let $R$ be any Poisson algebra. By identifying $a\otimes1, 1\otimes y_1, 1\otimes y_2\in R\otimes_{\bf k}{\bf k}[y_1,y_2]$ with $a, y_1, y_2\in R[y_1,y_2]$ respectively,
the Poisson algebra $R\otimes_{\bf k}{\bf k}[y_1,y_2]\cong R[y_1,y_2]$ is a Poisson double extension of $R$ with the PDE-data $\{q, 0, 0, w\}$. By Proposition~\ref{IteratedORE}(1),
the Poisson double extension $R[y_1,y_2]$ is an iterated  Poisson polynomial extension $R[y_1][y_2; \beta, \nu]_p$,
where
$$\begin{aligned}
&\beta(a)=0, &&\beta(y_1)=q_{12}y_1+w_2, \\
&\nu(a)=0, &&\nu(y_1)=q_{11}y_1^2+w_1y_1+w_0
\end{aligned}$$
for all $a\in R$.
\end{exam}

\begin{exam}
Let $R={\bf k}[x_1,\ldots,x_n]$ be a polynomial algebra and  $\mu=(\mu_{ij})$ be an $n\times n$-skew symmetric matrix with
entries in ${\bf k}$. Assume that $n\geq2$. By \cite[Example 4,5]{Good3},  $R$ becomes a Poisson algebra with Poisson bracket $\{x_i,x_j\}=\mu_{ij}x_ix_j$ for all $i,j$.
It is easy to see that
$R$ is a Poisson double extension of  Poisson subalgebra ${\bf k}[x_1,\cdots, x_{n-2}]$, since $R={\bf k}[x_1,\cdots,x_{n-2}][x_{n-1}; \alpha_1, 0]_p[x_n; \alpha_2, 0]_p$ is an iterated Poisson polynomial extension with $\nu_2=0$ and $\alpha_2(x_i)=\mu_{ni}x_i$ for $i=1,2,\cdots,n-1$, which satisfy the conditions of Proposition 4.3.

The Poisson algebra $\mathcal{O}(M_2({\bf k}))={\bf k}[a,b,c,d]$
with Poisson bracket
\[\begin{tabular}{lll}
$\{b,c\}=0$, &$\{b,a\}=-2ba$, &$ \{c,a\}=-2ca$,\\
$\{b,d\}=2bd$, &$\{c,d\}=2cd$, &$\{a,d\}=4bc$\end{tabular}\]
can be written in terms of  iterated Poisson polynomial extension
${\bf k}[b,c][a;\alpha_1,0]_p[d;\alpha_2,\nu_2]_p$ by \cite[2.4]{Oh8}. Since $\mathcal{O}(M_2({\bf k}))$ satisfies
the conditions of Proposition~\ref{itdoPe}, it is a Poisson double extension of ${\bf k}[b,c]$ with variables $a,d$ and a suitable PDE-data.
\end{exam}

%

\begin{exam}\label{DIM3}
Let $B$ be a connected graded right double extension of ${\bf k}[x]$ with $\deg x=\deg y_1=\deg y_2=1$. Then $B$ is generated by  $x,y_1,y_2$ subject to the relations
$$
\begin{aligned}
y_2y_1&=\mu_1 y_1^2+\mu_2y_1 y_2+\mu_3xy_1+\mu_4xy_2+\mu_5x^2,\\
y_1x&=\mu_6xy_1+\mu_7xy_2+\mu_8x^2,\\
y_2x&=\mu_9xy_1+\mu_{10}xy_2+\mu_{11}x^2,
\end{aligned}
$$
where $\mu_i\in{\bf k}$. In fact, $B$  is a connected graded Artin-Schelter regular algebra of global dimension 3. See \cite[Example 4.1]{ZhZh}.

Fix an element $\lambda\in{\bf k}\setminus\{0,1\}$ and set
$\Bbb F={\bf k}[t,t^{-1}]$. For all $i$ such that $1\leq i\leq 11$ except for $2,6,10$, there exist $f_i\in\Bbb F$ such that $f_i(\lambda)=\mu_i$ and $f_i(1)=0$ by Lagrange's Interpolation Formula \cite[Remark 0.6]{Row}. For $i=2, 6, 10$, choose $f_i\in\Bbb F$ such that $f_i(\lambda)=\mu_i$ and $f_i(1)=1$.
Let $A$ be an $\Bbb F$-algebra generated by $x,y_1,y_2$ subject to the relations
$$
\begin{aligned}
y_2y_1&=f_1 y_1^2+f_2y_1 y_2+f_3xy_1+f_4xy_2+f_5x^2,\\
y_1x&=f_6xy_1+f_7xy_2+f_8x^2,\\
y_2x&=f_9xy_1+f_{10}xy_2+f_{11}x^2.
\end{aligned}
$$
By Bergman's diamond lemma \cite{Be}, the set $\{x^iy_1^jy_2^k|i,j,k\geq0\}$ is an $\Bbb F$-basis and thus $A$ is a free left $\Bbb F[x]$-module with a basis $\{y_1^jy_2^k|j,k\geq0\}$. In particular, $t-1$ is a nonzero, nonunit and non-zero-divisor in $A$. Moreover, $A/(t-\lambda)A$ is isomorphic to the right double extension $B$
and $A/(t-1)A$  as a ring is isomorphic to the commutative polynomial ring ${\bf k}[x,y_1,y_2]$. Thus  the 5-tuple $(\{\lambda\}, \Bbb F, \Bbb F[x], A, t-1)$ satisfies Assumption~\ref{ASSUM}. It follows that $A_1:={\bf k}[x,y_1,y_2]$ is a Poisson algebra with Poisson bracket
$$
\begin{aligned}
\{y_2,y_1\}&=f_1'(1) y_1^2+f_2'(1)y_1 y_2+f_3'(1)xy_1+f_4'(1)xy_2+f_5'(1)x^2,\\
\{y_1,x\}&=f_6'(1)xy_1+f_7'(1)xy_2+f_8'(1)x^2,\\
\{y_2,x\}&=f_9'(1)xy_1+f_{10}'(1)xy_2+f_{11}'(1)x^2,
\end{aligned}
$$
where $f_i'(1)$ is the formal derivative of $f_i$ at $t=1$, and $B\cong A/(t-\lambda)A$ is a deformation of $A_1$ by Theorem~\ref{DPOE}. The Poisson algebra $A_1$ is a Poisson double extension of ${\bf k}[x]$ with
PDE-data $\{q,\alpha,\nu,w\}$, where
$$\begin{aligned}
q&=\{f_1'(1),f_2'(1)\}, &w&=\{f_3'(1)x,f_4'(1)x,f_5'(1)x^2\},\\
\alpha(x)&=\left(\begin{matrix} f_6'(1)x&f_7'(1)x\\ f_9'(1)x&f_{10}'(1)x\end{matrix}\right),&\nu(x)&=\left(\begin{matrix} f_8'(1)x^2\\ f_{11}'(1)x^2\end{matrix}\right).
\end{aligned}$$
\end{exam}

\begin{lem}\label{NORMA}
Let $A$ be a Poisson polynomial extension of a Poisson algebra $R$. Namely, $A=R[z;\beta, \nu]_p$.
If $z=z'+s$ for some $s\in R$, then $A=R[z';\beta,\nu']_p$ for some $\nu'$.
\end{lem}

\begin{proof}
 Define a ${\bf k}$-linear map $\nu'$ by
$$\nu':R\longrightarrow R,\ \ \nu'(r)=\beta(r)s+\nu(r)+\{r,s\}.$$
It is easy to check $\nu'$ is a derivation on $R$, since $\beta$, $\nu$ and $\{-,s\}$ are all derivations. Since $(\beta,\nu)$ satisfies (\ref{SKEPOI}) by \cite[Theorem 1.1]{Oh8}, it is checked easily that the pair $(\beta,\nu')$ also satisfies (\ref{SKEPOI}). Hence $A$ is the  Poisson polynomial extension $A=R[z';\beta,\nu']_p$ by \cite[Theorem 1.1]{Oh8}
\end{proof}

\medskip
In the next example, we see that there exists  a Poisson double extension that is not an iterated Poisson polynomial extension.

\begin{exam}\label{NONIT}
 We  consider a special case of  Example~\ref{DIM3}. Retain the notations of  Example~\ref{DIM3}. Let
$$\begin{aligned}
\mu_1&=\mu_3=\mu_4=\mu_6=\mu_8=\mu_{10}=\mu_{11}=0,\\
\mu_2&=-1, \ \ \mu_5=\mu_7=\mu_9=1
\end{aligned}$$
in Example~\ref{DIM3}.
That is, $B$ is the algebra generated by $x,y_1,y_2$ subject to the relations
$$\begin{aligned}
y_2y_1 &= -y_1y_2 + x^2,\\ y_1x &= xy_2, \\ y_2x &= xy_1 .
\end{aligned}$$

The corresponding $\Bbb F$-algebra $A$ is generated by $x,y_1,y_2$ subject to the relations
$$\begin{aligned}
y_2y_1 &= \left(-\frac{t-1}{\lambda-1}+\frac{t-\lambda}{1-\lambda}\right)y_1y_2 + \frac{t-1}{\lambda-1}x^2,\\
y_1x &=\frac{t-\lambda}{1-\lambda}xy_1+ \frac{t-1}{\lambda-1}xy_2,\\
 y_2x &= \frac{t-1}{\lambda-1}xy_1 + \frac{t-\lambda}{1-\lambda}xy_2.
\end{aligned}$$
Thus $A/(t-\lambda)A\cong B$ and the  corresponding Poisson double extension $A_1$ of ${\bf k}[x]$ is the Poisson algebra ${\bf k}[x,y_1,y_2]$  with the Poisson bracket
\begin{equation}\label{NIPPE1}
\{y_2,y_1\}  = -\frac{2}{\lambda-1}y_1y_2 + \frac{1}{\lambda-1}x^2,
\end{equation}
\begin{equation}\label{NIPPE2}
\{y_1,x\} =-\frac{1}{\lambda-1}xy_1+ \frac{1}{\lambda-1}xy_2,
\end{equation}
\begin{equation}\label{NIPPE3}
\{ y_2,x\} = \frac{1}{\lambda-1}xy_1 - \frac{1}{\lambda-1}xy_2.
\end{equation}

By setting $y=y_1-y_2$ and $z=y_1+y_2$,  the Poisson algebra $A_1$ can be identified with the Poisson algebra ${\bf k}[x, y, z]$ with Poisson bracket
\begin{equation}\label{ps}
\{x, y\}=\frac{2}{\lambda-1}xy,\; \{y, z\}=\frac{-2}{\lambda-1}x^2+\frac{-1}{\lambda-1}y^2+\frac{1}{\lambda-1}z^2, \; \{z, x\}=0.
\end{equation}
It is easy to compute that its modular derivation $D$  is given by
$$D(x)=\frac{2x}{\lambda-1},\  D(y)=\frac{2z-2y}{\lambda-1},\  D(z)=\frac{2y}{\lambda-1}.$$

We will show that $A_1$ is not an iterated Poisson polynomial extension over ${\bf k}[x]$. This phenomenon is similar to \cite[Example 4.1]{ZhZh}.
Suppose that $A_1$ is an iterated  Poisson polynomial extension over ${\bf k}[x]$. Then
 $$A_1={\bf k}[x][u;\beta',\nu']_p[w; \beta,\nu]_p.$$
 Since the Poisson subalgebra ${\bf k}[x][u;\beta',\nu']_p$ of $A_1$ is also a Poisson polynomial extension, by Lemma~\ref{NORMA}, we can assume that $u=g_0(x,z)z+\sum_{i=1}^{r}g_i(x,z)y^i$ for some $r\in \mathbb{N}$ and polynomials $g_i(x, z) \in {\bf k}[x, z]$. We also assume that not all the polynomials $g_i(x, z)$ are zero.
 Then by definition,
 $$\{u, x\}=\beta'(x)u+\nu'(x).$$
 Computing this equation by using \eqref{ps},
 \begin{equation}\label{eq}
  \sum_{i=1}^{r}\frac{-2i}{\lambda-1}xg_i(x,z)y^i=\beta'(x)\left(g_0(x,z)z+\sum_{i=1}^{r}g_i(x,z)y^i\right)+\nu'(x).
 \end{equation}
Since the left hand side  of this equation can be divided by $y$, this implies that
$$\beta'(x)g_0(x,z)z+\nu'(x)=0.$$
Namely, $\nu'(x)=\beta'(x)g_0(x,z)=0.$

If $\beta'(x)=0$, then $g_i(x,z)=0$ for all $i\geq 1$  by equation~\eqref{eq}. Hence, in this case,
\begin{equation}\label{case1}
u=g_0(x,z)z:=\sum_{i=1}^{m}f_i(x)z^i, \ \ \beta'=0, \ \ \nu'=0,
\end{equation}
for some $m\in \mathbb{N}$, and polynomials $f_i(x)\in {\bf k}[x]$.

If $g_0(x,z)=0$, then there exists $n\in \mathbb{N}$ such that $g_n(x,z)\neq 0$ by assumption.  And the equation~\eqref{eq} turns out to be
$$  \sum_{i=1}^{r}\left(\frac{-2i}{\lambda-1}x-\beta'(x)\right)g_i(x,z)y^i=0.$$
In this case, $\left(\frac{-2i}{\lambda-1}x-\beta'(x)\right)g_i(x,z)=0$  for all $i\geq 1$. Since $g_n(x,z)\neq 0$,
\begin{equation}\label{case2}
\beta'(x)=\frac{-2n}{\lambda-1}x,\  \nu'=0,\ \mbox{and} \ u=g_n(x, z)y^n:=p(x, z)y^n, \,
\end{equation}
where  $p(x, z)$ is a nonzero polynomial in  ${\bf k}[x, z]$.

For case \eqref{case1}, the Poisson structure of the Poisson subalgebra ${\bf k}[x][u;\beta',\nu']_p$ is trivial. By \cite[Theorem~3.5]{Wang},
$D(u)=-\beta(u)\in {\bf k}[x, u]$. And
$$\begin{aligned}
D(u)&=D\left(\sum_{i=1}^{m}f_i(x)z^i\right)=\sum_{i=1}^{m}D(f_i(x))z^i+f_i(x)iz^{i-1}D(z)\\
&=\sum_{i=1}^{m}\left(f_i'(x)z^i\frac{2x}{\lambda-1}+f_i(x)iz^{i-1} \frac{2y}{\lambda-1}\right).\end{aligned}$$
Since $D(u)=-\beta(u)\in {\bf k}[x, u]\subset {\bf k}[x, z]$, this implies that $\sum_{i=1}^{m}f_i(x)iz^{i-1}=0$, i.e.,
$ f_i(x)=0$ for all $i$.
This contradicts the assumption  $u\neq 0$.

For case \eqref{case2}, the modular derivation $\phi$ of  ${\bf k}[x][u;\beta',\nu']_p$ with respect to the volume form $\de\!x\wedge\de\!u$ is given by
$$\phi(x)=\frac{2nx}{\lambda-1}, \ \  \phi(u)=\frac{-2nu}{\lambda-1}.$$
Again, by  \cite[Theorem~3.5]{Wang}, $D(x)=\frac{2nx}{\lambda-1}-\beta(x)$, which implies that $\beta(x)=\frac{2(n-1)x}{\lambda-1}$. Moreover,
\begin{equation}\label{Du}
\beta(u)=\frac{-2nu}{\lambda-1}-D(u)=-\frac{\partial p(x,z)}{\partial x}\frac{2xy^n}{\lambda-1}-p(x,z)\frac{2ny^{n-1}z}{\lambda-1}-\frac{\partial p(x,z)}{\partial z}\frac{2y^{n+1}}{\lambda-1}.
\end{equation}

By equations~\eqref{case2} and \eqref{Du},
$$\begin{aligned}
\beta(\{u, x\}) &= \beta\left(\frac{-2n}{\lambda-1}xu\right)\\
&=\frac{-4nx}{(\lambda-1)^2}[(n-1) p(x,z)y^n-\frac{\partial p(x,z)}{\partial x}xy^n-np(x,z)y^{n-1}z-\frac{\partial p(x,z)}{\partial z}y^{n+1}],\end{aligned}$$
\begin{align*}
\{\beta(u), x\}&=\left\{-\frac{\partial p(x,z)}{\partial x}\frac{2xy^n}{\lambda-1}-p(x,z)\frac{2ny^{n-1}z}{\lambda-1}-\frac{\partial p(x,z)}{\partial z}\frac{2y^{n+1}}{\lambda-1}, x\right\}\\
&=-\frac{\partial p(x,z)}{\partial x}\frac{2x}{\lambda-1}\{y^n, x\}-p(x,z)\frac{2nz}{\lambda-1}\{y^{n-1},x\}-\frac{\partial p(x,z)}{\partial z}\frac{2}{\lambda-1}\{y^{n+1},x\} \\
&=\frac{\partial p(x,z)}{\partial x}\frac{4nx^2y^n}{(\lambda-1)^2}+p(x,z)\frac{4n(n-1)xy^{n-1}z}{(\lambda-1)^2}+\frac{\partial p(x,z)}{\partial z}\frac{4(n+1)xy^{n+1}}{(\lambda-1)^2} , \\
\end{align*}
and
$$
\{u, \beta(x)\}=\left\{u, \frac{2(n-1)}{\lambda-1}x\right\}=-\frac{4n(n-1)}{(\lambda-1)^2}xu=-\frac{4n(n-1)}{(\lambda-1)^2}p(x,z)xy^n.$$
By definition, $\beta$ is a Poisson derivation of Poisson subalgebra ${\bf k}[x][u;\beta',\nu']_p$,
which implies that
$$\beta(\{u, x\})=\{\beta(u),x\}+\{u, \beta(x)\}.$$
Hence $$np(x,z)y^{n-1}z=\frac{\partial p(x, z)}{\partial z}y^{n+1}.$$ This is absurd.
\end{exam}

\begin{exam}
For $0\neq q\in{\bf k}$, let $T_q$ be the ${\bf k}$-algebra given in \cite[Proposition 6.6]{JoSa}. That is, $T_q$ is a ${\bf k}$-algebra generated by $x,y,z$ subject to the relations
$$\begin{aligned}
yx &= q^{-1}xy+(q-q^{-1})z,\\
xz&=q^{-1}zx+(1-q^{-2})y,\\
yz&=qzy+(q^{-1}-q)x.
\end{aligned}$$
Observe that $T_q$ is a right double extension of ${\bf k}[z]$ with two variables  $x,y$ and DE-data $\{P,\sigma,\delta,\tau\}$, where
$$\begin{aligned}
P&=\{0, q^{-1}\},& \tau&=\{0, 0, (q-q^{-1})z\},\\
\sigma(z)&=\left(\begin{matrix} q^{-1}z&1-q^{-2}\\ q^{-1}-q&qz\end{matrix}\right),&\delta(z)&=\left(\begin{matrix} 0\\ 0\end{matrix}\right).
\end{aligned}$$

Set $\Bbb F:={\bf k}[t,t^{-1}]$. Replacing $q$ in $T_q$ with $t$, we get a right double extension $T_t$ of $\Bbb F[z]$ with variables $x,y$ and DE-data $\{P',\sigma',\delta',\tau'\}$, where
$$\begin{aligned}
P'&=\{0,t^{-1}\},& \tau'&=\{0,0, (t-t^{-1})z\},\\
\sigma'(z)&=\left(\begin{matrix} t^{-1}z&1-t^{-2}\\ t^{-1}-t&tz\end{matrix}\right),&\delta'(z)&=\left(\begin{matrix} 0\\ 0\end{matrix}\right).
\end{aligned}$$
Set $\Lambda:={\bf k}\setminus\{0,1\}$. By \cite[Proposition 6.6]{JoSa}, $T_q$ is an integral domain and thus $(\Lambda,\Bbb F, \Bbb F[z], T_t, t-1)$ satisfies Assumption~\ref{ASSUM}. Hence
$T_1:=T_t/(t-1)T_t\cong {\bf k}[z,x,y]$ is a Poisson ${\bf k}$-algebra with Poisson bracket
$$\begin{aligned}
\{y,x\} &= -xy+2z,\\
\{x,z\}&=-zx+2y,\\
\{y,z\}&=zy-2x
\end{aligned}$$
appearing in \cite[Example 4.4]{JoOh}, which is a Poisson double extension of ${\bf k}[z]$ with variables $x,y$ and  PDE-data $\{p,\alpha,\nu,w\}$, where
$$\begin{aligned}
p&=\{0,-1\},&
\alpha(z)&=\left(\begin{matrix}-z&2\\ -2&z\end{matrix}\right),&\nu(z)&=\left(\begin{matrix}0\\ 0\end{matrix}\right),&w&=\{0,0,2z\}.
\end{aligned}$$
Moreover, $T_q\cong T_t/(t-q)T_t$ is a deformation of $T_1$ by Theorem~\ref{DPOE}.
\end{exam}

\begin{exam}
Let $0\neq h\in{\bf k}$. Let $B(h)$ be the  graded double extension given in \cite[Example 4.2]{ZhZh}.
That is, $B(h)$ is the graded algebra generated by $x_1,x_2,y_1,y_2$ subject to the relations
$$\begin{array}{lcccccr}
x_2x_1=-x_1x_2&&&&&&\\
y_2y_1=-y_1y_2&&&&&&\\
y_1x_1=h(x_1y_1&+&x_2y_1&+&x_1y_2&&)\\
y_1x_2=h(&&&+&x_1y_2&&)\\
y_2x_1=h(&+&x_2y_1&&&&)\\
y_2x_2=h(&-&x_2y_1&-&x_1y_2&+&x_2y_2).
\end{array}$$
The last four relations are associated to the matrix
$$\Sigma=h\left(\begin{matrix}1&1&1&0\\0&0&1&0\\0&1&0&0\\0&-1&-1&1
\end{matrix}\right)$$
and $B(h)$ is a connected graded Artin-Schelter regular algebra of global dimension 4 as observed in \cite[Example 4.2 and Corollary 4.7]{ZhZh}.

Fix  an element $\lambda\in{\bf k}\setminus\{0,1\}$ and set $\Bbb F:={\bf k}[t,t^{-1}]$. Let $A$ be an ${\Bbb F}$-algebra generated by $x_1,x_2, y_1,y_2$ subject to relations
$$\begin{array}{lcccccr}
x_2x_1=f_1(t)x_1x_2&&&&&&\\
y_2y_1=f_2(t)y_1y_2&&&&&&\\
y_1x_1=f_3(t)x_1y_1&+&f_4(t)x_2y_1&+&f_5(t)x_1y_2&&\\
y_1x_2=&+&f_6(t)x_2y_1&+&f_7(t)x_1y_2&&\\
y_2x_1=&+&f_8(t)x_2y_1&+&f_9(t)x_1y_2&&\\
y_2x_2=&+&f_{10}(t)x_2y_1&+&f_{11}(t)x_1y_2&+&f_{12}(t)x_2y_2,
\end{array}$$
where $f_i(t)\in\Bbb F$ such that
$$\begin{aligned}
f_1(\lambda)&=-1, &f_1(1)&=1;&f_2(\lambda)&=-1, &f_2(1)&=1;
\end{aligned}$$
$$\left(\begin{matrix}f_3(\lambda)&f_4(\lambda)&f_5(\lambda)&0\\0&f_6(\lambda)&f_7(\lambda)&0\\0&f_8(\lambda)&f_9(\lambda)&0\\0&f_{10}(\lambda)&f_{11}(\lambda)&f_{12}(\lambda)
\end{matrix}\right)=\Sigma,
$$
$$\left(\begin{matrix}f_3(1)&f_4(1)&f_5(1)&0\\0&f_6(1)&f_7(1)&0\\0&f_8(1)&f_9(1)&0\\0&f_{10}(1)&f_{11}(1)&f_{12}(1)
\end{matrix}\right)=\left(\begin{matrix}1&0&0&0\\0&1&0&0\\0&0&1&0\\0&0&0&1
\end{matrix}\right).
$$
(Such $f_i(t)$'s exist by Lagrange's Interpolation Formula. See \cite[Remark 0.6]{Row}.)
Observe that $A$ is a free $\Bbb F$-module with basis $\{x_1^ix_2^jy_1^ky_2^\ell|i,j,k,\ell\geq0\}$ by Bergman's diamond lemma \cite{Be}.
Let $R$ be  the $\Bbb F$-subalgebra of $A$ generated by $x_1, x_2$. Note that $R$ is isomorphic to the $\Bbb F$-algebra generated by $x_1,x_2$  subject to the relation
$$x_2x_1=f_1(t)x_1x_2,$$
 that $A$ is a free left $R$-module with basis  $\{y_1^ky_2^\ell|k,\ell\geq0\}$
 and that
$A/(t-1)A$ is commutative.  Moreover note that the 5-tuple $(\{\lambda\},\Bbb F, R,A, t-1)$ satisfies
Assumption~\ref{ASSUM}.  Hence $A_1:=A/(t-1)A={\bf k}[x_1,x_2,y_1,y_2]$ is a Poisson algebra with Poisson bracket
$$\begin{array}{lcccccr}
\{x_2, x_1\}=f_1'(1)x_1x_2,&&&&&& \\
\{y_2,y_1\}=f_2'(1)y_1y_2,&&&&&&\\
\{y_1,x_1\}=f_3'(1)x_1y_1&+&f_4'(1)x_2y_1&+&f_5'(1)x_1y_2,&&\\
\{y_1,x_2\}=&+&f_6'(1)x_2y_1&+&f_7'(1)x_1y_2,&&\\
\{y_2,x_1\}=&+&f_8'(1)x_2y_1&+&f_9'(1)x_1y_2,&&\\
\{y_2,x_2\}=&+&f_{10}'(1)x_2y_1&+&f_{11}'(1)x_1y_2&+&f_{12}'(1)x_2y_2,
\end{array}$$
where $f_i'(1)$ is the formal derivatives of $f_i(t)$ at $t=1$,
and $B(h)\cong A/(t-\lambda)A$ is a deformation of $A_1$ by Theorem~\ref{DPOE}.
Observe that $A_1={\bf k}[x_1,x_2,y_1,y_2]$ is a Poisson double extension of ${\bf k}[x_1,x_2]$ with an appropriate PDE-data $\{q,\alpha,\nu,w\}$.
\end{exam}

\noindent
{\bf Acknowledgments:} All the authors would like to thank the referees for their helpful suggestions and comments.
The first author is supported by the NSFC (key project 11331006). The second author is supported by National  Research Foundation of Korea,  NRF-2017R1A2B4008388, and thanks the Korea Institute for Advanced Study for the warm hospitality during the preparation of this paper.
The third author is supported by the NSFC (projects 11301180 and 11771085).



\bibliographystyle{amsplain}

\providecommand{\bysame}{\leavevmode\hbox to3em{\hrulefill}\thinspace}
\providecommand{\MR}{\relax\ifhmode\unskip\space\fi MR }
\providecommand{\MRhref}[2]{%
  \href{http://www.ams.org/mathscinet-getitem?mr=#1}{#2}
}
\providecommand{\href}[2]{#2}

\end{document}